\documentclass{article}

\usepackage{PRIMEarxiv}
\usepackage[utf8]{inputenc}
\usepackage[T1]{fontenc}
\usepackage{fullpage}
\usepackage{graphicx}
\graphicspath{{media/}}
\usepackage{amsmath,amssymb,amsfonts,amsthm,mathtools}
\usepackage{booktabs}
\usepackage{nicefrac}
\usepackage{microtype}
\usepackage{lipsum}
\usepackage{authblk}
\usepackage{color}
\usepackage{float}
\usepackage[inline]{enumitem}
\usepackage{tikz}
\usepackage{dsfont}
\usetikzlibrary{arrows}
\theoremstyle{definition}
\newtheorem{defi}{Definition}[section]
\theoremstyle{plain}
\newtheorem{theo}{Theorem}[section]
\newtheorem{prop}[theo]{Proposition}
\newtheorem{lem}[theo]{Lemma}

\theoremstyle{remark}
\newtheorem{rem}{Remark}[section]

\theoremstyle{definition}

\usepackage{fancyhdr}
\fancypagestyle{firstpage}{
  \fancyhf{}
  \fancyfoot[L]{\begin{minipage}{0.95\textwidth}\raggedright\footnotesize
  \end{minipage}}

}
\usepackage{hyperref}
\usepackage[backend=biber,style=numeric,sorting=none,maxnames=99]{biblatex}

\begin{document}
\title{The Maker-Breaker directed triangle game}
\author{
  \textbf{Hrishikesh Jagtap} \\
  Department of Mathematics and Computer Science, Eindhoven University of Technology, Eindhoven, Netherlands
 \\
  \texttt{h.p.jagtap@tue.nl}
  \and
  \textbf{Moumanti Podder} \\
  Department of Mathematics, Indian Institute of Science Education and Research, Pune, Maharashtra, India
 \\
  \texttt{moumanti@iiserpune.ac.in}
}

\maketitle
\thispagestyle{firstpage}

\begin{abstract}
    In this work, we investigate Maker-Breaker directed triangle games, a directionally constrained variant of the classical Maker-Breaker triangle game. Our board of interest is a tournament, and the winning sets constitute all directed triangles (i.e.\ $3$-cycles) present in the tournament. We begin by studying the Maker-Breaker directed triangle game played on a specially defined tournament called the parity tournament, and we identify the board size threshold to be $n=7$, which is to say that if the size (i.e.\ the number of vertices) of the parity tournament equals $n$, Breaker has a winning strategy for $3\le n< 7$, while Maker can ensure a win for herself for $n\ge 7$. For the $(1:b)$ biased version of this game, wherein Breaker is allowed to claim $b$ directed edges of the tournament in each of her turns while Maker is allowed to claim only one directed edge in each of hers, we prove that the bias threshold $b^*(n)$ satisfies $\sqrt{\left(1/12+o(1)\right)\ n}\le b^{*}(n) \le\sqrt{\left(8/3+o(1)\right)\ n}$, which matches the order of magnitude (i.e.\ $\sqrt{n}$) of the bias threshold for the undirected counterpart of this game. Next, we consider the game on random tournaments $T(n,p)$, wherein the vertices are labeled $1,2,\ldots,n$, and the edge between $i$ and $j$, for each $i<j$, is directed from $i$ towards $j$ with probability $p$, independent of all else. We prove that Maker wins this game with probability approaching $1$ as $n \to \infty$ for any fixed $p \in (0,1)$. Extending the notion of `bias' from undirected games to our directed framework, we introduce the flip-biased Maker-Breaker directed triangle game on the parity tournament with flip budget $\kappa(n)$, where we allow Breaker to strategically flip the directions of at most $\kappa(n)$ edges of the tournament before the game begins (once again, $n$ indicates the size of the tournament). We show that the flip-bias threshold $\kappa^*(n)$ for this game is of order $n^2$. More precisely, for odd $n\ge 11$ we show $n(n-11)/12 \le \kappa^*(n) \le (n^2-1)/8$, and for even $n \ge 14$ we show $n(n-14)/12 + 1\le \kappa^*(n) \le n^2/8+ n/4-1$.
\end{abstract}

\section{Introduction}

Let $X$ be a finite set which we call the \emph{board}, and let $\mathcal{F}$ be a family of subsets of $X$ referred to as a \emph{family of winning sets} (the elements of $\mathcal{F}$ are referred to as the \emph{winning sets}).  A \emph{positional game}, indicated by $(X,\mathcal{F)}$, is a \emph{two-player combinatorial game} played on the board $X$, with the two players taking turns to claim yet-unclaimed elements of $X$. The game ends when all elements of $X$ have been claimed. When we focus on the special class of games called the \emph{Maker-Breaker positional games}, these two players are called, as the name suggests, \emph{Maker} and \emph{Breaker}. Maker's objective is to claim all the elements of at least one winning set $F \in \mathcal{F}$, while Breaker tries to prevent her from doing so. Maker wins this game if she is able to accomplish her objective before the end of the game, and Breaker wins otherwise. We shall assume that Maker starts the game by playing the first round. The board $X$ as well as the family $\mathcal{F}$ of winning sets is revealed to the two players \emph{before} the game begins, making it a \emph{perfect information game}, and the players have complementary goals, thus implying that there can be no draw. We refer the reader to \cite{beck2008combinatorial},\cite{Hefetz2014-mg} and \cite{krivelevich2014positional} for an introduction to the key results and topics in this field.

A special class of such games is the Maker-Breaker positional game \emph{on graphs}, where the underlying board $X$ is the edge-set, $E$, of a given graph $G=(V,E)$, and the family $\mathcal{F}$ of winning sets comprises the subsets of edges contributing to the formation of a prespecified structure. In other words, Maker's objective in such a game is to ensure that the subgraph induced by the edges claimed by her contains this prespecified structure. For instance, $\mathcal{F}$ may be the set of all triangles (\cite{chvatal1978biased}, \cite{balogh2011chvatal}, \cite{glazik2022new}), all $k$-cliques for a fixed but arbitrary positive integer $k$ (\cite{beck2008combinatorial}, \cite{bednarska2000biased}), all spanning trees (\cite{chvatal1978biased}, \cite{beck1982remarks}), all Hamiltonian cycles (\cite{beck1985random}, \cite{krivelevich2011critical}) etc. 

So far, we have assumed that each player claims a single edge during each of her turns, and such a version of these games is referred to as an \emph{unbiased game}. However, many of these unbiased games happen to be easy wins for Maker. To level the playing field, we may instead allow Breaker to claim $b$ edges during each of her turns, for some pre-decided positive integer $b$, while allowing Maker to claim a single edge in each of her turns. More generally, we consider the \emph{biased} $(a:b)$ Maker-Breaker positional game, where Maker claims exactly $a$ edges in each of her turns, and Breaker claims exactly $b$ edges in each of hers. In particular, $a=b=1$ represents the unbiased game discussed above. Going forward, we shall assume a game to be unbiased unless stated otherwise. A natural question that arises in such biased games, particularly for the $(1:b)$  biased game, is as follows: what is the minimum value $b^{*}(n)$ of the bias $b$, where $n$ is the size of the graph $G$ (i.e.\ $|V|=n$), at which the game transitions from being a Maker's win to a Breaker's win (i.e.\ for $b < b^{*}(n)$, Maker wins, whereas for $b\ge b^{*}(n)$, Breaker wins)? The value $b^*(n)$ is called the \emph{bias threshold} for that game. 

Our work considers a framework with additional \emph{directional constraints} on the winning sets of the traditional Maker-Breaker positional game played on graphs. In particular, we study a `directionally constrained' variant of the usual triangle game (\cite{chvatal1978biased}, \cite{balogh2011chvatal}, \cite{glazik2022new}), i.e.\ we consider directed triangles (or $3$-cycles) as our winning sets, and our board is a \emph{tournament} (i.e.\ a complete directed graph) instead of an undirected graph. For the unbiased version of this directed triangle game on the \emph{parity tournament} $\Pi(n)$ (a special type of tournament which is regular for odd $n$ and near-regular for even $n$), we proved that for $n\le 6$, Breaker has a winning strategy for the game played on $\Pi(n)$,while for all $n \ge 7$, Maker has a winning strategy. We also identify that the bias threshold for this directed triangle game on the parity tournament is of the order $\sqrt{n}$, matching the order of its undirected counterpart (as seen in \cite{chvatal1978biased}), and provide an upper and lower bound on $b^*(n)$.

Moving away from the parity tournament, we address the Maker-Breaker directed triangle game played on a random tournament $T(n,p)$ in which, for every pair $i,j$ of vertices with $i<j$, the edge between $i$ and $j$ is oriented, independent of all else, \emph{from} $i$ \emph{towards} $j$ with probability $p$, and \emph{from} $j$ \emph{towards} $i$ with probability ($1-p$), where $p\in (0,1)$ is a given constant. We show that Maker wins with probability tending to $1$ as the board size (i.e.\ the size of the tournament, $n$) tends to infinity, in this particular game.

Finally, we introduce a new notion of bias, called the \emph{flip-bias}, for the directed triangle game on the parity tournament, whereby Breaker is allowed to preemptively and strategically flip the orientations of at most $\kappa(n)$ edges \emph{before} the game begins. These perturbations allow Breaker some additional advantage on the resultant board. For the parity tournament of size $n$, we define the \emph{flip-bias threshold} $\kappa^*(n)$ to be the minimum number of edge flips required by Breaker to guarantee a win on the perturbed board (that is, if $\kappa(n) < \kappa^*(n)$, then Maker has a winning strategy, while if $\kappa(n) \ge \kappa^*(n)$ then Breaker has a winning strategy). We prove that $\kappa^*(n)$ grows quadratically with $n$. More precisely, for odd $n\ge 11$ we show $\frac{n(n-11)}{12} \le \kappa^*(n) \le \frac{n^2-1}{8}$, and for even $n\ge 14$ we show $\frac{n(n-14)}{12} + 1\le \kappa^*(n) \le \frac{n^2}{8}+\frac{n}{4}-1$.

\subsection{Literature Review}

Erdős and Selfridge in \cite{erdos1973combinatorial} introduced the notion of Maker-Breaker games, along with the use of potential functions for game analysis, and the formulation of derandomisation techniques to construct explicit winning strategies.  The Erdős-Selfridge Theorem from \cite{erdos1973combinatorial} is a cornerstone result that provided a criterion for Breaker's win in unbiased games, which kicked off a cascade of very exciting results in this field of work. Later, Chvátal and Erdős in  \cite{chvatal1978biased} introduced bias into the game framework by allowing Breaker to claim additional moves in each of her turns, thereby proposing a natural question of identifying the bias threshold. Over the past few decades, J\'{o}zsef Beck, via his many papers (\cite{beck1981positional}, \cite{beck1982remarks}, \cite{beck1985random}, \cite{beck1994deterministic}) and an exhaustive monograph (\cite{beck2008combinatorial}), has proved many insightful results and generalisations of the previous work. These works address a variety of winning substructures that Maker aims to capture, like triangles, Hamiltonian cycles, spanning trees, cliques of fixed sizes, etc. For an overarching introduction to the topic of positional games and its main results, we refer the reader to \cite{Hefetz2014-mg}.
The triangle game was previously studied by Chvátal and Erdős in \cite{chvatal1978biased}, where they proved that the corresponding bias threshold lies between $\sqrt{2n}$ and $2\sqrt{n}$, where $n$ is the size of the board. Balogh and Samotij in \cite{balogh2011chvatal} improved the upper bound from \cite{chvatal1978biased} to $1.958\sqrt{n}$ using a randomised strategy. Recently, Glazik and Srivastav in \cite{glazik2022new} further improved this upper bound to $1.633\sqrt{n}$. Tightening these bounds remains a problem that still eludes us.

Our motivation for addressing the Maker-Breaker directed triangle game stems from the intention of introducing directional constraints to the winning sets and analysing similar threshold phenomena and potential differences with their undirected counterparts.
Games of a similar flavour have been addressed in  \cite{frieze2021maker}, which analysed the strong connectivity game and the Hamiltonicity game played on a \emph{complete digraph}.
Other related games played on undirected boards, but with some directional aspect incorporated into them, include the \emph{tournament game} (studied in \cite{clemens2016random} and  \cite{clemens2015remark}), where the board is an undirected complete graph $K_n$ or the Erd\H{o}s -R\'{e}nyi random graph  $G(n,p)$: whenever Maker claims an undirected edge, she also assigns it a direction, and Maker wins if her set of claimed directed edges contains a prespecified tournament $T_k$. Another class of games with directional constraints are referred to as the \emph{orientation games} ( \cite{bollobas1998oriented},  \cite{ben2012biased},  \cite{clemens2017non}), in which players alternately orient the undirected edges (that have not been previously oriented by either player) of $K_n$ to produce a single final tournament: Maker wins if that final tournament has a property $P$ (which is concerned with the presence of a prespecified winning substructure).

\subsection{Notation and terminology}\label{Subsection:Notation}

Our board of interest for the games discussed in this work are \emph{tournaments}, which are complete graphs endowed with \emph{directed edges.} Given a vertex set $V_{n}=\{1,2,\ldots,n\}$, we denote by $(i,j)$ the \emph{directed edge} that goes \emph{from} the vertex $i$ \emph{to} the vertex $j$, for all distinct $i,j \in V_{n}$. 

\begin{defi}[Tournament]\label{defi:Tournament}
     A \emph{complete digraph} or a \emph{tournament}, $T(n)$, on $V_{n}$ is a graph in which precisely one of the two directed edges, $(i,j)$ and $(j,i)$, is present, for all distinct $i,j \in V_{n}$, there are no self-loops nor any parallel edges.
\end{defi}

We denote by $E(T(n))$ the set of all directed edges present in $T(n)$. For any three distinct vertices $i, j, k \in V_{n}$, we say that $(i,j,k)$ forms a \emph{directed triangle} if the directed edges $(i,j)$, $(j,k)$ and $(k,i)$ are all present in $E(T(n))$.

We can now define the \emph{unbiased Maker-Breaker directed triangle game} on a tournament.

\begin{defi}[(1:1) Maker–Breaker directed triangle game]\label{defi: MB directed triangle game}
Let $T(n)$ be a tournament on the vertex set given by $V_n$. The board on which the Maker-Breaker directed triangle game is played is the edge set $E(T(n))$ of the tournament, and the winning sets are the collection of triples of edges corresponding to all possible directed triangles in $T(n)$. Each winning set comprises the three edges that form the respective directed triangle. We can describe the unbiased (1:1) Maker-Breaker directed triangle game on any such tournaments $T(n)$ as follows:
\begin{enumerate}
    \setlength{\itemsep}{0pt}
    \item Maker plays the first move and claims one edge from the tournament;
    \item Breaker plays next and deletes one edge from the tournament;
    \item this goes on alternately until all the edges of the tournament have been exhausted;
    \item the winner is:\\- Maker, if she manages to claim a directed triangle by the end of the game, \\- Breaker, if she prevents Maker from claiming such a directed triangle.
\end{enumerate}
\end{defi}
We will assume that any game referred to here onward is unbiased and that Maker is the first player by default, unless stated otherwise.

In particular, we study different \emph{types} of tournaments as our underlying boards for the directed triangle game.
To begin with, we first focus on \emph{parity tournaments} $\Pi(n)$ in which the direction of each edge is dictated by a certain parity-based rule. 

\begin{defi}[Parity tournament]\label{defi:Parity rule}
    In the parity tournament $\Pi(n)$, for any two vertices $i, j \in V_n$ with $i < j$, the direction of the edge between $i$ and $j$ is defined as follows:
    \begin{itemize}
        \item If $(i+j)$ is odd, then $(i,j) \in E(\Pi(n))$ (i.e., the edge is directed from $i$ to $j$).
        \item If $(i+j)$ is even, then $(j,i) \in E(\Pi(n))$ (i.e., the edge is directed from $j$ to $i$).
    \end{itemize}
\end{defi}

\begin{defi}[Transitive tournament]\label{defi:transitive tournament}
A tournament is called \emph{transitive} if there exists an ordering of its vertices
$(v_1,\ldots,v_n)$ such that for every $1\le i<j\le n$, the edge between $v_i$ and $v_j$
is oriented from $v_i$ to $v_j$.
Equivalently, $T(n)$ contains no directed cycle.

We write $\Lambda(n)$ for the transitive tournament on $V_n$ in which every edge is oriented
from the smaller label to the larger label, that is, $(i,j)\in E(\Lambda(n))$ if and only if $i<j$.
\end{defi}

\begin{defi}[Random tournament]\label{defi:randomtournament}
Fix $p\in(0,1)$. The random tournament $T(n,p)$ on vertex set $V_n$ is obtained by independently orienting each edge
$\{i,j\}$ with $i<j$ from $i$ to $j$ with probability $p$, and from $j$ to $i$ with probability $1-p$.
\end{defi}

\section{Main Results} \label{section:Main_Results}

In this section, we state our main findings concerning the Maker–Breaker directed triangle game. Our results can naturally divide into three parts. In the first part we consider the game played on the parity tournament, in the second part we study the game on a random tournament with constant edge-orientation probability, and in the third part we introduce the \emph{flip-bias} mechanism, which endows Breaker with additional power by allowing her to flip the orientations of at most $\kappa(n)$ edges before the game begins.

\subsection{Maker–Breaker directed triangle game on the parity tournament}\label{subsec:paritytournament}

We now state our results for the parity tournament $\Pi(n)$ from Definition~\ref{defi:Parity rule}.
Our first theorem addresses the size threshold for the board to turn the game into a Maker's win in the unbiased case and the second one identifies the bounds on the bias threshold for the biased $(1:b)$ game, both considered on the parity tournament.

\begin{theo}\label{thm:thresholdunbiased}
For the Maker–Breaker directed triangle game played on the parity tournament $\Pi(n)$:
\begin{itemize}
    \item If $n\le 6$, then Breaker has a winning strategy.
    \item If $n \ge 7$, then Maker has a winning strategy.
\end{itemize}
\end{theo}

This result identifies $n=7$ as the threshold; in particular, the tournament $\Pi(7)$ is the smallest instance for which Maker's winning strategy exists. We use a \emph{cycle hopping strategy}, which will be described in detail when we present the proof, to outline Maker's winning strategy for $n\ge7$. For $n \le 6$, Breaker can use some variants of a pairing strategy to ensure herself a win.

\begin{theo}\label{thm:bias-threshold}
Let $b^*(n)$ denote the threshold bias for the $(1:b)$ Maker–Breaker directed triangle game on $\Pi(n)$. Then, for sufficiently large $n$, 
\begin{equation} \label{eqn:bias threshold inequality}
    \sqrt{\left(\frac{1}{12} + o(1)\right)n} \ \le\  b^*(n) \ \le\  \sqrt{\left(\frac{8}{3} + o(1)\right)n}.
\end{equation}

\end{theo}

\subsection{Maker's win on random tournaments}\label{subsec:randomtournament}

Let $T(n,p)$ be the random tournament from Definition~\ref{defi:randomtournament}, with $p\in(0,1)$ fixed.
Our next theorem establishes that Maker wins the directed triangle game on $T(n,p)$ with high probability.

\begin{theo}\label{thm:random}
For any fixed $p\in (0,1)$, the Maker–Breaker directed triangle game played on the random tournament $T(n,p)$ is a Maker's win with probability tending to $1$ as $n\to\infty$.
\end{theo}

The proof of Theorem \ref{thm:random} uses the second moment method for the random variable counting induced copies of $\Pi(7)$ in $T(n,p)$. Since Maker has a winning strategy on $\Pi(7)$ by Theorem \ref{thm:thresholdunbiased}, the existence of one induced copy of $\Pi(7)$ inside $T(n,p)$ is enough to give Maker a win on $T(n,p)$.

Generally, on a tournament $T(n)$, one can define a notion of how "balanced" it is based on the in-degrees and out-degrees of its vertices. Tournaments in which each vertex has equal (when $n$ is odd) or near equal (when $n$ is even) in-degree and out-degree for each vertex are called \emph{regular or near-regular tournaments}. Regular tournaments can be thought of as being the most balanced. At the other end of the spectrum, we have \emph{transitive tournaments}, which are tournaments where the vertices can be ordered such that every edge is directed from a vertex indexed "higher" in the order to a vertex indexed "lower" in the order. Transitive tournaments are the most imbalanced.

\begin{rem}
    The out-degree of a vertex is called its \emph{score}. The scores of vertices in $V_n$ play a crucial role in determining the number of directed triangles (or $3$- cycles) in the corresponding tournament. Particularly, as described in \cite{moon2015topics}, we know that a regular or near-regular tournament contains the maximum number of directed triangles for a tournament of that size. In contrast, a transitive tournament is acyclic (see Definition \ref{defi:transitive tournament}).

\end{rem}

\subsection{Flip-Bias Threshold}\label{subsec:flipbias}

To counteract Maker's advantage in the unbiased game on the parity tournament $\Pi(n)$, we allow Breaker to flip the orientations of at most $\kappa(n)$ edges before the game begins, and then the players start the usual unbiased directed triangle game on the resulting perturbed tournament.
We write $\kappa^*(n)$ for the smallest integer $k$ for which Breaker can choose a set $S$ of flips, with $|S|\le k$, that guarantees her a win on the perturbed board.

\begin{theo}\label{thm:flip-bias}
    The flip-bias threshold $\kappa^*(n)$ satisfies $\kappa^*(n) \le (n^2-1)/8$ if $n$ is odd, and $\kappa^*(n) \le n^2/8+n/4-1$ if $n$ is even.
    Furthermore, if $n$ is odd and $n\ge 11$ then $\kappa^*(n) \ge n(n-11)/12$,
    and if $n$ is even and $n\ge 14$ then $\kappa^*(n) \ge n(n-14)/12 + 1$.
\end{theo}

\begin{prop}\label{prop:linear-triangles-breakerwin}
    For every $n$, there exists a tournament on $n$ vertices with at most $\lfloor n/3\rfloor$ directed triangles on which Breaker has a winning strategy in the unbiased directed triangle game.
\end{prop}

Taken together, Theorems \ref{thm:thresholdunbiased}–\ref{thm:flip-bias} form the core of our results on the Maker–Breaker directed triangle game. The first couple of results establish the critical size and bias thresholds for the game on the parity tournament, the third shows Maker's win on random tournaments for any fixed $p\in(0,1)$, and the fourth introduces the flip-bias mechanism and its threshold, thereby quantifying the extra power Breaker must use to counteract Maker's advantage in the unbiased directed triangle game on the parity tournament. 

\section{Proofs of our results stated in Subsection \ref{subsec:paritytournament}}\label{sec: proofs1}

To facilitate the proofs for the theorems stated in Section \ref{section:Main_Results}, we first prove some results regarding the structure and count of the directed triangles in a parity tournament. 

\subsection{Structural characterization of directed triangles -- our winning sets}\label{subsec:structuralcharacterization}

Recall that the parity tournament $\Pi(n)$ is defined on the vertex set $V_n$. A triple $(a,b,c)$ is said to be a directed triangle in a tournament if $(a,b), (b,c),(c,a)$ are present in $E(T(n))$. Any cyclic rotation of this triple will represent the same underlying directed triangle, and thus we define its canonical representation to be the one in which the smallest element is written first; that is, if $a = \min\{a,b,c\}$, then we take $(a,b,c)$ as the unique canonical representative of the triangle.

Below, we shall establish necessary and sufficient conditions for such a canonical triple to correspond to a directed triangle in the parity tournament. Particularly, we will prove that these conditions force the three vertices to be arranged in ascending order and have alternating parity.

\begin{lem}\label{lem:triangle_structure}
Let $a,b,c \in V_n$ with $a = \min\{a,b,c\}$. The directed triangle $(a,b,c)$ with edges $(a,b),(b,c),(c,a)$ exists in the parity tournament $\Pi(n)$ (as defined in Definition \ref{defi:Parity rule}) if and only if $a+b \text{ is odd},\, b+c \text{ is odd},\, a+c \text{ is even}$ and $a < b < c$.

\end{lem}

\begin{proof}
In order for $(a,b,c)$ to form a directed triangle in the parity tournament, we require the following (as dictated by the parity rule described in Definition \ref{defi:Parity rule}):
\begin{enumerate}
    \item Since the edge $(a,b)$ exists and we know that $a<b$, it must be that $a+b$ is odd.
    \item Similarly, the edge $(c,a)$ exists where $a<c$, and thus $a+c$ must be even.
    \item Now that we know $a + b$ is odd and $a + c$ is even, we can conclude that $a$ and $b$ have opposite parity, while $a$ and $c$ have the same parity. Consequently, $b$ and $c$ have opposite parity, which implies that $b+c$ is odd. Our directed triangle also requires the edge $(b,c)$ to exist, and thus, it must be that $b<c$.
\end{enumerate}
Conversely, let us assume that $a<b<c$ and that $a+b \text{ is odd},\, b+c \text{ is odd},\, a+c \text{ is even}.$
Then, the parity rule implies the following:
\begin{enumerate}
    \item Since $a+b$ is odd and $a<b$, the edge between $a$ and $b$ is oriented as $(a,b)$.
    \item Since $b+c$ is odd and $b<c$, the edge between $b$ and $c$ is oriented as $(b,c)$.
    \item Since $a+c$ is even and $a<c$, the edge between $a$ and $c$ is oriented as $(c,a)$.
\end{enumerate}
These three directed edges together form the directed triangle $(a,b,c).$ This completes the proof.
\end{proof}

\subsection{Counting the winning sets} \label{subsec: count of winning sets}

Having characterised the structure of directed triangles in $\Pi(n)$, it is only natural to count them next. Many criteria determining whether $\Pi(n)$ is a Maker's win or a Breaker's win rely heavily on the count of these winning sets.

\begin{lem}\label{lem:countwinning}
    In a parity tournament $\Pi(n)$, the number of directed triangles $w(n)$ is $\sum_{i=1}^{n-2} \left\lceil \frac{i}{2} \right\rceil \cdot \left\lceil \frac{n - (i+1)}{2} \right\rceil$.
\end{lem}
\begin{proof}
    Fix $b \in \{2, \dots, n-1\}$. In order to come up with triples $(a,b,c)$ that satisfy the conditions stated in Lemma \ref{lem:triangle_structure}, we must choose
    \begin{itemize}
        \item $a$ from $\{1, \dots, b-1\}$ with parity opposite to $b$, and
        \item $c$ from $\{b+1, \dots, n\}$ with parity opposite to $b$ as well, so that $a$ and $c$ match each other’s parity.
    \end{itemize}
    
    The number of integers in $\{1, \dots, b-1\}$ with parity different from $b$ is $\left\lceil \frac{b-1}{2} \right\rceil.$
    Similarly, the number of integers in $\{b+1, \dots, n\}$ with parity different from \( b \) is $\left\lceil \frac{n-b}{2} \right\rceil.$
    Finally, summing over $b$ we get
    \begin{align}
    w(n) = \sum_{b=2}^{n-1} \left\lceil \frac{b-1}{2} \right\rceil \cdot \left\lceil \frac{n-b}{2} \right\rceil
          = \sum_{i=1}^{n-2} \left\lceil \frac{i}{2} \right\rceil \cdot \left\lceil \frac{n-(i+1)}{2} \right\rceil.\nonumber
    \end{align}
    
    If $n=2m+1$ is odd, write $i=2k$ and $i=2k+1$. Then
    \begin{align*}
        w(n)
        &=\sum_{k=1}^{m-1}\left\lceil \frac{2k}{2}\right\rceil \left\lceil \frac{2m+1-(2k+1)}{2}\right\rceil
        +\sum_{k=0}^{m-1}\left\lceil \frac{2k+1}{2}\right\rceil \left\lceil \frac{2m+1-(2k+2)}{2}\right\rceil\\
        &=\sum_{k=1}^{m-1} k(m-k)+\sum_{k=0}^{m-1}(k+1)(m-k)
        =\frac{n^3-n}{24}.
    \end{align*}
    
    If $n=2m$ is even, the same parity split gives
    \begin{align*}
        w(n)
        &=\sum_{k=1}^{m-1}\left\lceil \frac{2k}{2}\right\rceil \left\lceil \frac{2m-(2k+1)}{2}\right\rceil
        +\sum_{k=0}^{m-2}\left\lceil \frac{2k+1}{2}\right\rceil \left\lceil \frac{2m-(2k+2)}{2}\right\rceil\\
        &=\sum_{k=1}^{m-1} k(m-k)+\sum_{k=0}^{m-2}(k+1)(m-k-1)
        =\frac{n^3-4n}{24}.
        \end{align*}
\end{proof}

As expected, this count also aligns with a classical result from \cite{moon2015topics} concerning $3$-cycles in arbitrary tournaments.
\begin{prop} [Theorem $4$ from \cite{moon2015topics}] \label{theo: moon's count}
    Let $c_3(T(n))$ denote the number of $3$-cycles in the tournament $T(n)$.
    If $(s_1,s_2,\ldots,s_n)$ is the score vector of $T(n)$, where $s_i$ is the number of outgoing edges from vertex $i$, then
    \[
    c_3(T(n))
    = \binom{n}{3}
      - \sum_{i=1}^{n} \binom{s_i}{2}
    \ \le\ 
    \begin{cases}
    \dfrac{1}{24}\bigl(n^{3}-n\bigr), & \text{if $n$ is odd},\\[6pt]
    \dfrac{1}{24}\bigl(n^{3}-4n\bigr), & \text{if $n$ is even}.
    \end{cases}
    \]
    Equality holds if and only if $T(n)$ is regular when $n$ is odd, and near-regular when $n$ is even.
\end{prop}

Accordingly, $\Pi(n)$ attains this upper bound: it is regular when $n$ is odd and near-regular when $n$ is even. Hence the value of $w(n)$ obtained in Lemma \ref{lem:countwinning} agrees exactly with Proposition \ref{theo: moon's count}.

\subsection{\texorpdfstring{Breaker's winning strategy on $\Pi(n)$ for $n < 7$}{Breaker's winning strategy on Pi(n) for n < 7}}

To begin with, we address the Maker-Breaker directed triangle game played on $\Pi(n)$ for $n=3,4,5$. Since $\Pi(3)$, $\Pi(4)$ and $\Pi(5)$ occur as induced sub-tournaments of $\Pi(6)$, it suffices to prove Breaker's win on $\Pi(6)$ (Lemma~\ref{lem: Breaker wins on Pi(6)}).

\begin{lem} \label{lem: breaker monotonicity}
If Breaker has a winning strategy in the Maker-Breaker directed triangle game on $\Pi(6)$, then she also has a winning strategy for the game played on $\Pi(n)$, where $n=3,4,5$.
\end{lem}
\begin{proof}
For each $m\in\{3,4,5\}$, the sub-tournament of $\Pi(6)$ induced by the vertex set $\{1,2,\dots,m\}$ is exactly $\Pi(m)$, since the parity rule depends only on the order relation and the parities of the labels. In particular, every directed triangle in $\Pi(m)$ is also a directed triangle in this induced copy of $\Pi(m)$ inside $\Pi(6)$.

Assume for contradiction that Maker has a winning strategy on $\Pi(m)$ for some $m\in\{3,4,5\}$. Then, when the game is played on $\Pi(6)$, Maker can simply restrict all her moves to the induced copy of $\Pi(m)$ and follow that winning strategy there. Any edge deleted by Breaker outside this induced copy has no effect on the play inside it, and any edge deleted by Breaker inside it is already accounted for by Maker's winning strategy on $\Pi(m)$. Hence Maker would also win on $\Pi(6)$, contradicting the assumption that Breaker wins on $\Pi(6)$.

Therefore, if Breaker wins on $\Pi(6)$, then she also wins on $\Pi(m)$ for each $m\in\{3,4,5\}$.

\end{proof}

We will now prove that Breaker in fact has a winning strategy for the game played on $\Pi(6)$. Before we do so, we first introduce a few tools that we will need for the proof.

\begin{defi}[Pairing strategy]\label{defi:pairingstrategy}
    A pairing strategy for Breaker begins with finding a collection of pairwise disjoint pairs of edges of the board such that each directed triangle contains, fully, one of these pairs. We refer to the two edges, $e_{1}$ and $e_{2}$, in any given pair $\{e_1,e_2\}$, as \emph{siblings}, and such a pair is called a \emph{blocking pair}. The strategy prescribes Breaker to, in each of her moves, respond by deleting the sibling of the edge claimed by Maker in the previous turn. In case Maker's previously claimed edge was not a part of any pair, Breaker arbitrarily deletes any edge in the tournament. 
\end{defi}
Because Breaker always deletes the sibling of whichever edge Maker claims, Maker can never claim both edges of any pair. Recall, from Definition \ref{defi:pairingstrategy}, that each of our winning sets (i.e.\ each directed triangle present in our tournament) contains precisely one of these pairs of edges, and one of the siblings belonging to each such pair is always deleted by Breaker. Consequently, Maker cannot claim any of the winning sets entirely, ensuring a win for Breaker. 

We refer to the phenomenon where Breaker prevents Maker from claiming a directed triangle, say $(a,b,c)$, by deleting one of the siblings belonging to the pair, say $\{e_{1},e_{2}\}$, that is contained in $(a,b,c)$, as "$\{e_{1},e_{2}\}$ \emph{blocking} the directed triangle $(a,b,c)$". 

As demonstrated in Figure \ref{fig: edgepairsontournaments} (here, the paired directed edges are indicated by connecting them using dashed curves), such a pairing strategy can be very useful for determining the outcome of certain games. Another way to represent such pairings, and, more generally, the winning sets in a game, is via a suitable hypergraph. This approach affords us a much easier visualization of the board as well as the progression of the game. For instance, the winning sets arising out of $\Pi(5)$, as well as a possible pairing illustrated in Figure \ref{fig: edgepairsontournaments}, are more readily seen via the hypergraph representation shown in Figure \ref{fig: edgepairsonhypergraph}. Particularly, each hyperedge in this hypergraph representation corresponds to the set of 3 directed edges that form a directed triangle in the underlying tournament, and any element in such a hyperedge corresponds to a directed edge from the underlying tournament.

\begin{defi}[Game hypergraph and cut operation]\label{defi:gamehypergraph}
Given a tournament $T(n)$, we write $\mathcal H(T(n))=(X,\mathcal F)$ for the $3$-uniform hypergraph whose vertex set is $X=E(T(n))$ and whose hyperedges are precisely the winning sets of the directed triangle game on $T(n)$, that is, the triples of directed edges that form directed triangles in $T(n)$. For any edge $e\in X$, we define
\[
\operatorname{Cut}(\mathcal H(T(n)),e)
:=
\bigl(X\setminus\{e\},\{F\in\mathcal F:e\notin F\}\bigr),
\]
which is the game hypergraph obtained after deleting the edge $e$ from the board.

\end{defi}

\begin{figure}[htbp]
    \centering
    \includegraphics[width=0.5\textwidth]{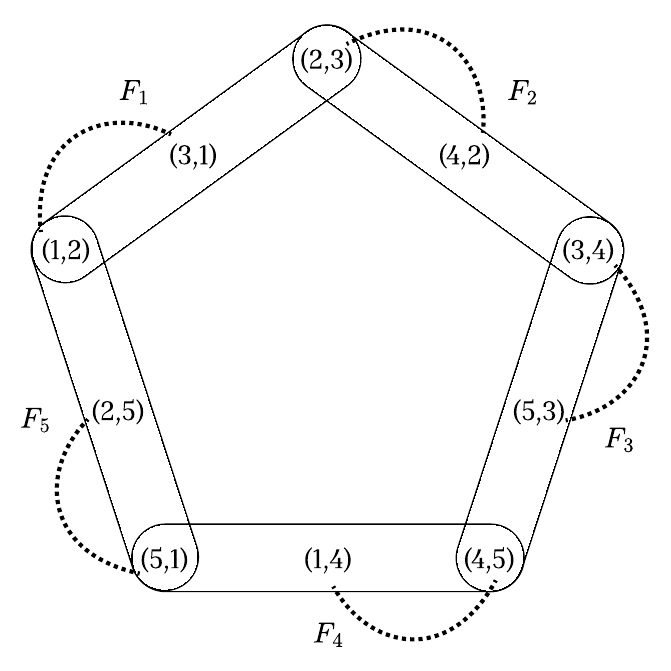}
    \caption{Pairing on $\Pi(5)$'s hypergraph representation}
    \label{fig: edgepairsonhypergraph}
\end{figure}

\begin{figure}[htbp]
\centering
\begin{tikzpicture}[>=stealth, thick, scale=2.25]
    \node[circle, draw, minimum size=8pt] (v1) at (90:1.5)  {1};
    \node[circle, draw, minimum size=8pt] (v2) at (162:1.5) {2};
    \node[circle, draw, minimum size=8pt] (v3) at (234:1.5) {3};
    \node[circle, draw, minimum size=8pt] (v4) at (306:1.5) {4};
    \node[circle, draw, minimum size=8pt] (v5) at (18:1.5)  {5};
    \path[->]
      (v1) edge
        node[midway, above, sloped] {$a_1$}
      (v2);
    \path[->]
      (v3) edge
        node[midway, above, sloped] {$a_2$}
      (v1);
    \path[->]
      (v2) edge
        node[midway, above, sloped] {$a_3$}
      (v3);
    \path[->]
      (v4) edge
        node[midway, below, sloped] {$a_4$}
      (v2);
    \path[->]
      (v3) edge
        node[midway, below, sloped] {$a_5$}
      (v4);
    \path[->]
      (v5) edge
        node[midway, above, sloped] {$a_6$}
      (v3);
    \path[->]
      (v4) edge
        node[midway, below, sloped] {$a_7$}
      (v5);
    \path[->]
      (v1) edge
        node[midway, above, sloped] {$a_8$}
      (v4);
    \path[->]
      (v5) edge
        node[midway, right, sloped] {$a_9$}
      (v1);
    \path[->]
      (v2) edge
        node[midway, above, sloped] {$a_{10}$}
      (v5);
    \path (v1) -- node[midway, draw=none, inner sep=0pt] (A1) {} (v2);
    \path (v3) -- node[midway, draw=none, inner sep=0pt] (A2) {} (v1);
    \path (v2) -- node[midway, draw=none, inner sep=0pt] (A3) {} (v3);
    \path (v4) -- node[midway, draw=none, inner sep=0pt] (A4) {} (v2);
    \path (v3) -- node[midway, draw=none, inner sep=0pt] (A5) {} (v4);
    \path (v5) -- node[midway, draw=none, inner sep=0pt] (A6) {} (v3);
    \path (v4) -- node[midway, draw=none, inner sep=0pt] (A7) {} (v5);
    \path (v1) -- node[midway, draw=none, inner sep=0pt] (A8) {} (v4);
    \path (v5) -- node[midway, draw=none, inner sep=0pt] (A9) {} (v1);
    \path (v2) -- node[midway, draw=none, inner sep=0pt] (A10) {} (v5);
    \draw[dashed, bend right=20] (A1) to (A2);
    \draw[dashed, bend left=15] (A3) to (A4);
    \draw[dashed, bend right=25] (A5) to (A6);
    \draw[dashed, bend left=20] (A7) to (A8);
    \draw[dashed, bend right=15] (A9) to (A10);
\end{tikzpicture}
\caption{Diagram of $\Pi(5)$ with edges labeled $a_1,\dots,a_{10}$ and dashed arcs showing Breaker's pairing strategy.}
\label{fig: edgepairsontournaments}
\end{figure}

\subsection{\texorpdfstring{Breaker's winning strategy on $\Pi(6)$}{Breaker's winning strategy on Pi(6)}}

\begin{lem}\label{lem: Breaker wins on Pi(6)}
Breaker wins the Maker-Breaker directed triangle game on $\Pi(6)$.
\end{lem}

\begin{figure}
\centering
\begin{tikzpicture}[->,thick,scale=1.5, every node/.style={circle, draw, minimum size=10pt}]
    \node (1) at (0,2) {1};
    \node (2) at (1.73,1) {2};
    \node (3) at (1.73,-1) {3};
    \node (4) at (0,-2) {4};
    \node (5) at (-1.73,-1) {5};
    \node (6) at (-1.73,1) {6};
    \draw (1) edge (2);
    \draw (3) edge (1);
    \draw (1) edge (4);
    \draw (5) edge (1);
    \draw (1) edge (6);
    \draw (2) edge (3);
    \draw (4) edge (2);
    \draw (2) edge (5);
    \draw (6) edge (2);
    \draw (3) edge (4);
    \draw (5) edge (3);
    \draw (3) edge (6);
    \draw (4) edge (5);
    \draw (6) edge (4);
    \draw (5) edge (6);
\end{tikzpicture}
\caption{$\Pi(6)$}
\label{fig:Tp6}
\end{figure}

\begin{figure}[htbp]
\centering
  \includegraphics[width=0.5\textwidth]{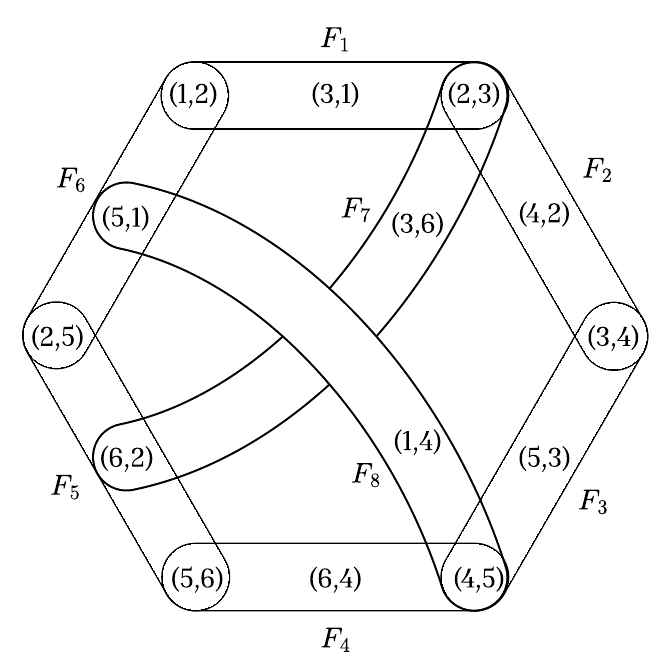}
  \caption{$\Pi(6)$'s hypergraph representation}
  \label{fig:tp6hyper}
\end{figure}

\begin{proof}
In the parity tournament $\Pi(6)$, there are $8$ winning sets (i.e.\ directed triangles), listed as follows:
\[
\begin{array}{rcl}
F_1 &=& \{(1,2),\,(3,1),\,(2,3)\},\\[1mm]
F_2 &=& \{(2,3),\,(4,2),\,(3,4)\},\\[1mm]
F_3 &=& \{(3,4),\,(5,3),\,(4,5)\},\\[1mm]
F_4 &=& \{(4,5),\,(6,4),\,(5,6)\},\\[1mm]
F_5 &=& \{(5,6),\,(6,2),\,(2,5)\},\\[1mm]
F_6 &=& \{(2,5),\,(5,1),\,(1,2)\},\\[1mm]
F_7 &=& \{(2,3),\,(3,6),\,(6,2)\},\\[1mm]
F_8 &=& \{(5,1),\,(1,4),\,(4,5)\}.
\end{array}
\]

We split into two cases according to Maker's first move.

\medskip
\noindent\textbf{Case 1: $m_1\neq (2,5)$.}
Breaker plays $b_1=(2,5)$. Then $F_5$ and $F_6$ are removed, and the surviving winning sets are
\[
\mathcal F'=\{F_1,F_2,F_3,F_4,F_7,F_8\}.
\]
We begin with the two pairwise disjoint blocking families
\[
\mathcal P_L=
\Bigl\{
\{(1,2),(3,1)\},
\{(3,4),(4,2)\},
\{(4,5),(5,3)\},
\{(5,6),(6,4)\},
\{(3,6),(6,2)\},
\{(1,4),(5,1)\}
\Bigr\}
\]
and
\[
\mathcal P_R=
\Bigl\{
\{(1,2),(3,1)\},
\{(2,3),(4,2)\},
\{(3,4),(5,3)\},
\{(5,6),(6,4)\},
\{(3,6),(6,2)\},
\{(1,4),(5,1)\}
\Bigr\}.
\]
Each family is pairwise disjoint and each surviving winning set in $\mathcal F'$ contains exactly one displayed pair.

\smallskip
\noindent\textbf{Subcase 1A: $m_1\in\{(1,6),(2,3),(1,2),(3,1),(4,2),(3,4),(3,6),(6,2)\}$.}
Breaker uses $\mathcal P_L$ when $m_1\in\{(1,6),(2,3)\}$. For the remaining possibilities, Breaker modifies exactly one pair of $\mathcal P_L$ as follows:
\[
\begin{array}{c|c}
m_1 & \text{replace in }\mathcal P_L \\ \hline
(1,2) & \{(1,2),(3,1)\}\text{ by }\{(3,1),(2,3)\}\\[1mm]
(3,1) & \{(1,2),(3,1)\}\text{ by }\{(1,2),(2,3)\}\\[1mm]
(4,2) & \{(3,4),(4,2)\}\text{ by }\{(2,3),(3,4)\}\\[1mm]
(3,4) & \{(3,4),(4,2)\}\text{ by }\{(2,3),(4,2)\}\\[1mm]
(3,6) & \{(3,6),(6,2)\}\text{ by }\{(2,3),(6,2)\}\\[1mm]
(6,2) & \{(3,6),(6,2)\}\text{ by }\{(2,3),(3,6)\}
\end{array}
\]
In every row, the resulting family is pairwise disjoint, avoids Maker's already-claimed edge $m_1$, and blocks all winning sets in $\mathcal F'$.

\smallskip
\noindent\textbf{Subcase 1B: $m_1\in\{(1,4),(5,1),(5,3),(4,5),(6,4),(5,6)\}$.}
Breaker uses $\mathcal P_R$ when $m_1=(4,5)$. For the remaining possibilities, Breaker modifies exactly one pair of $\mathcal P_R$ as follows:
\[
\begin{array}{c|c}
m_1 & \text{replace in }\mathcal P_R \\ \hline
(1,4) & \{(1,4),(5,1)\}\text{ by }\{(4,5),(5,1)\}\\[1mm]
(5,1) & \{(1,4),(5,1)\}\text{ by }\{(1,4),(4,5)\}\\[1mm]
(5,3) & \{(3,4),(5,3)\}\text{ by }\{(3,4),(4,5)\}\\[1mm]
(6,4) & \{(5,6),(6,4)\}\text{ by }\{(4,5),(5,6)\}\\[1mm]
(5,6) & \{(5,6),(6,4)\}\text{ by }\{(4,5),(6,4)\}
\end{array}
\]
Again, in every row, the resulting family is pairwise disjoint, avoids Maker's already-claimed edge $m_1$, and blocks all winning sets in $\mathcal F'$.

Therefore, in Case~1 Breaker can follow the pairing strategy from Definition \ref{defi:pairingstrategy} from her second move onward, and Maker can never claim all three edges of a surviving winning set.

\medskip
\noindent\textbf{Case 2: $m_1=(2,5)$.}
Breaker plays $b_1=(6,2)$.

\smallskip
\noindent\textbf{Subcase 2A: $m_2=(5,1)$.}
Breaker must play $b_2=(1,2)$, since otherwise Maker wins immediately on
$F_6=\{(2,5),(5,1),(1,2)\}$.
After these four moves, the surviving winning sets are
\[
\mathcal F''=\{F_2,F_3,F_4,F_8\}.
\]
Breaker now uses the pairwise disjoint blocking family
\[
\Bigl\{
\{(2,3),(4,2)\},
\{(3,4),(5,3)\},
\{(5,6),(6,4)\},
\{(1,4),(4,5)\}
\Bigr\},
\]
which blocks $F_2,F_3,F_4,F_8$ respectively. Hence Breaker wins by the pairing strategy.

\smallskip
\noindent\textbf{Subcase 2B: $m_2\neq(5,1)$.}
Breaker plays $b_2=(5,1)$.
Then the surviving winning sets are
\[
\mathcal F'''=\{F_1,F_2,F_3,F_4\}.
\]
For the possible values of $m_2$, Breaker uses the following pairwise disjoint blocking families:
\[
\begin{array}{c|l}
m_2 & \text{Blocking family} \\ \hline
(2,3) & \bigl\{\{(1,2),(3,1)\},\{(3,4),(4,2)\},\{(4,5),(5,3)\},\{(5,6),(6,4)\}\bigr\}\\[1mm]
(3,1),(1,4),(1,6),(3,6) & \bigl\{\{(1,2),(2,3)\},\{(3,4),(4,2)\},\{(4,5),(5,3)\},\{(5,6),(6,4)\}\bigr\}\\[1mm]
(1,2) & \bigl\{\{(2,3),(3,1)\},\{(3,4),(4,2)\},\{(4,5),(5,3)\},\{(5,6),(6,4)\}\bigr\}\\[1mm]
(4,2) & \bigl\{\{(1,2),(3,1)\},\{(2,3),(3,4)\},\{(4,5),(5,3)\},\{(5,6),(6,4)\}\bigr\}\\[1mm]
(3,4) & \bigl\{\{(1,2),(3,1)\},\{(2,3),(4,2)\},\{(4,5),(5,3)\},\{(5,6),(6,4)\}\bigr\}\\[1mm]
(4,5) & \bigl\{\{(1,2),(3,1)\},\{(2,3),(4,2)\},\{(3,4),(5,3)\},\{(5,6),(6,4)\}\bigr\}\\[1mm]
(5,3) & \bigl\{\{(1,2),(3,1)\},\{(2,3),(4,2)\},\{(3,4),(4,5)\},\{(5,6),(6,4)\}\bigr\}\\[1mm]
(6,4) & \bigl\{\{(1,2),(3,1)\},\{(2,3),(4,2)\},\{(3,4),(5,3)\},\{(4,5),(5,6)\}\bigr\}\\[1mm]
(5,6) & \bigl\{\{(1,2),(3,1)\},\{(2,3),(4,2)\},\{(3,4),(5,3)\},\{(4,5),(6,4)\}\bigr\}
\end{array}
\]
Each displayed family avoids Maker's already-claimed edge $m_2$, is pairwise disjoint, and blocks all winning sets in $\mathcal F'''$.
Hence Breaker wins in Subcase~2B as well.

Thus Breaker has a winning strategy in all cases, and therefore wins the Maker-Breaker directed triangle game on $\Pi(6)$.

\end{proof}

\subsection{\texorpdfstring{Maker's winning strategy on $\Pi(7)$}{Maker's winning strategy on Pi(7)}}

From this point onward, let
\[
\mathcal H:=\mathcal H(\Pi(7)).
\]
Our approach is to identify, after Breaker's first move, a cyclic sequence of hyperedges in which consecutive hyperedges overlap in a single element. Maker will then move only on these overlap elements and force a win.

\begin{defi}[Threat cycle] \label{defi:cycle}
Let $\mathcal H=(X,\mathcal F)$ be a $3$-uniform hypergraph.
A \emph{threat cycle} of length $t$ in $\mathcal H$ is a cyclic sequence
\[
(x_1,y_1,x_2,y_2,\ldots,x_t,y_t,x_1),
\]
such that the $2t$ displayed elements $x_1,y_1,\ldots,x_t,y_t$ are pairwise distinct, and
\[
\{x_i,y_i,x_{i+1}\}\in\mathcal F
\qquad\text{for each }i=1,\ldots,t,
\]
where $x_{t+1}:=x_1$.
We say that the threat cycle is \emph{based at} $x_1$.
\end{defi}

\begin{defi}[Main outer cycle] \label{defi:mainoutercycle}
The sequence
\[
C_0=((1,2),(3,1),(2,3),(4,2),(3,4),(5,3),(4,5),(6,4),(5,6),(7,5),(6,7),(1,6),(7,1),(2,7),(1,2))
\]
is called the \emph{main outer threat cycle}.

\end{defi}
This main outer threat cycle is illustrated in Figure \ref{fig:outercycle} by the hyperedges coloured grey.

\begin{figure}[htbp]
\centering
  \includegraphics[width=0.5\textwidth]{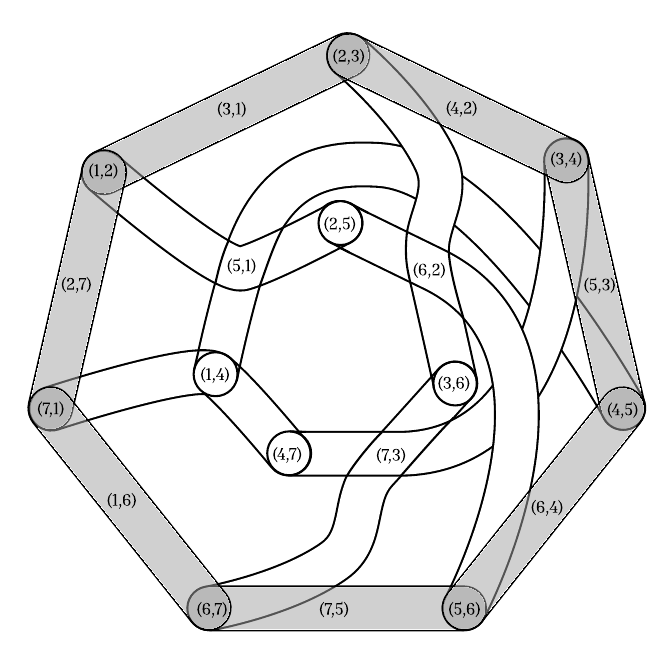}
  \caption{The Main outer cycle $C$}
  \label{fig:outercycle}
\end{figure}

\begin{defi}[Internal bridges]\label{defi: internal bridges}
In addition to $C_0$, the hypergraph contains the following internal paths:
\[
\begin{array}{rcl}
B_1 &:& ((1,2),(5,1),(2,5),(6,2),(5,6)),\\[1mm]
B_2 &:& ((2,3),(6,2),(3,6),(7,3),(6,7)),\\[1mm]
B_3 &:& ((3,4),(7,3),(4,7),(1,4),(7,1)),\\[1mm]
B_4 &:& ((4,5),(5,1),(1,4),(4,7),(7,1)).
\end{array}
\]

\end{defi}
Figure \ref{fig:bridges} illustrates these four internal bridges.

\begin{figure}[H]
\centering
  \includegraphics[width=0.5\textwidth]{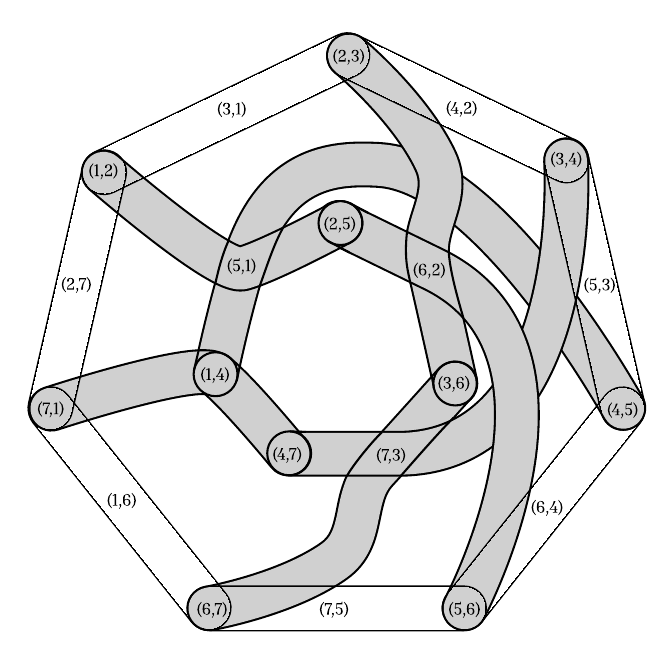}
  \caption{Internal bridges}
  \label{fig:bridges}
\end{figure}

\begin{lem}\label{lem:cycle_exists}
If $m_1=(1,2)$, then after Breaker deletes her first edge $b_1$, the surviving hypergraph
$\mathcal H'=\operatorname{Cut}(\mathcal H,b_1)$ contains a threat cycle based at $(1,2)$.
\end{lem}
\begin{proof}
We split according to Breaker's first move.

\textbf{Case 1: $b_1$ is not on the main outer threat cycle $C_0$.}
Then none of the elements of $C_0$ is removed, so $C_0$ itself survives as a threat cycle based at $(1,2)$.

\textbf{Case 2: $b_1$ lies on $C_0$.}
We distinguish two subcases.

\smallskip
\textbf{Subcase 2A: $b_1=(5,6)$.}
Then the sequence
\[
C_1=((1,2),(3,1),(2,3),(6,2),(3,6),(7,3),(6,7),(1,6),(7,1),(2,7),(1,2))
\]
is a surviving threat cycle based at $(1,2)$, since its successive triples are winning sets and none of its elements equals $(5,6)$.

\smallskip
\textbf{Subcase 2B: $b_1\in C_0\setminus\{(5,6)\}$.}
Partition the possible choices of $b_1$ into
\[
L=\{(2,7),(7,1),(1,6),(6,7),(7,5)\}
\quad\text{and}\quad
R=\{(3,1),(2,3),(4,2),(3,4),(5,3),(4,5),(6,4)\}.
\]

If $b_1\in L$, then
\[
C_2=((1,2),(3,1),(2,3),(4,2),(3,4),(5,3),(4,5),(6,4),(5,6),(6,2),(2,5),(5,1),(1,2))
\]
is a surviving threat cycle based at $(1,2)$.

If $b_1\in R$, then
\[
C_3=((1,2),(2,7),(7,1),(1,6),(6,7),(7,5),(5,6),(6,2),(2,5),(5,1),(1,2))
\]
is a surviving threat cycle based at $(1,2)$.

In each case, a threat cycle based at $(1,2)$, avoiding $b_1$, persists.
\end{proof}

\subsection{Maker's winning strategy via cycle-hopping}

Having established that a threat cycle based at $(1,2)$ persists in $\mathcal{H}' = \operatorname{Cut}(\mathcal{H}, b_1)$, regardless of Breaker’s first move, we now describe how Maker uses this cycle to win the game.

\begin{lem}\label{lem:cycle_hopping}
    Provided that Maker claims the element $m_1=(1,2)$, she is able to secure a win for herself in the Maker-Breaker directed triangle game on $\Pi(7)$.
\end{lem}

\begin{proof}
Maker first claims $x_1:=(1,2)$.
After Breaker's first move $b_1$, Lemma \ref{lem:cycle_exists} gives a surviving threat cycle
\[
C=(x_1,y_1,x_2,y_2,\ldots,x_t,y_t,x_1)
\]
based at $x_1=(1,2)$.

Maker now claims $x_2,x_3,\ldots,x_t$ in this order.

Suppose $2\le r\le t-1$, and Maker has just claimed $x_r$.
Then the hyperedge $\{x_{r-1},y_{r-1},x_r\}$ contains two Maker-claimed elements, namely $x_{r-1}$ and $x_r$.
If Breaker does not delete $y_{r-1}$ on her next move, then Maker claims $y_{r-1}$ on the following turn and wins.
Hence Breaker is forced to delete $y_{r-1}$.

Thus, after Maker has claimed $x_2,\dots,x_{t-1}$, Breaker has been forced to delete
$y_1,\dots,y_{t-2}$ in order.

Now Maker claims $x_t$.
At this moment, both hyperedges
\[
\{x_{t-1},y_{t-1},x_t\}
\qquad\text{and}\qquad
\{x_t,y_t,x_1\}
\]
are immediate threats: each already contains two Maker-claimed elements.
Breaker can delete at most one of $y_{t-1}$ and $y_t$.
Maker then claims the other one on her next move and completes a winning set.

Therefore Maker wins on $\Pi(7)$.
\end{proof}

Lemma \ref{lem:cycle_exists} and Lemma \ref{lem:cycle_hopping} together establish Maker's winning strategy on $\Pi(7)$.

\subsection{\texorpdfstring{Maker's winning strategy on $\Pi(n)$, for all $n > 7$}{Maker's winning strategy on Pi(n), for all n > 7}}

\begin{lem}[Extension to larger tournaments]\label{lem:extension}
Maker has a winning strategy in the Maker-Breaker directed triangle game on $\Pi(n)$ for all $n\ge 7$.
\end{lem}
\begin{proof}
Fix the vertex set $V_7=\{1,2,3,4,5,6,7\}\subset V_n$.
The sub-tournament induced by $V_7$ is exactly $\Pi(7)$.
Maker restricts all her moves to this induced copy and follows her winning strategy from Lemma \ref{lem:cycle_hopping}.

Any edge deleted by Breaker outside this induced copy has no effect on the play inside it.
Any edge deleted by Breaker inside it is handled exactly as in the game on $\Pi(7)$.
Hence the restricted game inside $V_7$ is a copy of the already solved game on $\Pi(7)$, and Maker wins there.

Therefore Maker wins on $\Pi(n)$ for every $n\ge 7$.

\end{proof}
Lemma \ref{lem:cycle_exists}, Lemma \ref{lem:cycle_hopping}, and Lemma \ref{lem:extension} together prove Theorem \ref{thm:thresholdunbiased}.

This result addresses the game played on a parity tournament, which is regular when $n$ is odd and near-regular when $n$ is even. Intuitively, one could argue that for an arbitrary tournament, which might be less regular, the corresponding game might be easier for Breaker to win, since the number of winning sets would be smaller in the case of more irregular and imbalanced score sequences. Thus, we could guess the board size threshold for an arbitrary tournament to be greater than or equal to $7$. However, it remains to find the value of board size thresholds for tournaments with given arbitrary score sequences.

\subsection{\texorpdfstring{Bounds on $b^*(n)$}{Bounds on b*(n)}}

 We consider a biased variant of the Maker-Breaker directed triangle game on the parity tournament. Particularly, we address the $(1:b)$ Maker-Breaker directed triangle game on $\Pi(n)$ with the intention of identifying bounds on the bias threshold, $b^*(n)$, such that for $b\ge b^*(n)$, Breaker has a winning strategy and for $b<b^*(n)$, Maker has a winning strategy.
 
 We use the following criterion for Maker's win, to derive a lower bound on $b^*(n)$:
 
\begin{theo}[See \cite{beck2008combinatorial}]\label{thm:BeckMakerCriterion}
    Let $\mathcal{F}$ be a family of subsets of a finite set $X$. Suppose $a$ and $b$ are positive integers, and define
    \[
    \Delta_2(\mathcal{F}) \ =\  \max\Bigl\{\,\Bigl|\{A\in\mathcal{F}:\{u,v\}\subset A\}\Bigr| : u,v\in X,\  u\neq v \Bigr\}.
    \]
    If
    \begin{equation}
      \sum_{A\in \mathcal{F}} \left(\frac{a+b}{a}\right)^{-|A|} \ >\  \frac{a^2\,b^2}{(a+b)^3}\,\Delta_2(\mathcal{F})\,|X|,
     \label{Beck'sMakerinequality}  
    \end{equation}
    then Maker (as the first player) has a winning strategy for the $(a:b)$ biased Maker-Breaker game on $(X,\mathcal{F})$.
\end{theo}

In our biased $(1:b)$ directed triangle game on $\Pi(n)$, $a=1$, and each winning set $F \in \mathcal{F}$ has size $|F| = 3$. From the structure of directed triangles, we observe that $\Delta_2(\mathcal{F}) = 1$, since no two edges can appear together in more than one triangle (because specifying two edges in a directed triangle uniquely determines the third). The elements in $X$ correspond to the directed edges of $\Pi(n)$, which is our board. Thus, $|X| = \binom{n}{2}$.

Let $w(n)$ denote the number of directed triangles in $\Pi(n)$. 
Given these conditions, Inequality \ref{Beck'sMakerinequality} from Theorem \ref{thm:BeckMakerCriterion} would boil down to
\begin{align}
    w(n)\,(1+b)^{-3}  >  \frac{b^2\,|X|}{(1+b)^3},
 \end{align}
which gives us
\begin{align}
    b  <  \sqrt{\frac{w(n)}{|X|}}.\label{eqn:1}
\end{align}
Thus, for all $b$ satisfying this inequality, Maker will have a winning strategy in the corresponding $(1:b)$ Maker-Breaker directed triangle game. Equivalently, the bias threshold $b^*(n)$ must satisfy
\[
b^*(n) \ \ge\  \sqrt{\frac{w(n)}{|X|}}.
\]
We have 
\[
w(n)
\ =\ 
\begin{cases}
\displaystyle \frac{n^3 - n}{24}, & \text{if $n$ is odd},\\[6pt]
\displaystyle \frac{n^3 - 4n}{24}, & \text{if $n$ is even},
\end{cases}
\]
and
\[
|X| = \frac{n(n-1)}{2}.
\]
It follows that
\[
\sqrt{\frac{w(n)}{|X|}}
\ =\ 
\begin{cases}
\displaystyle \sqrt{\frac{n+1}{12}} = \displaystyle \sqrt{\left(\frac{1}{12} +o(1)\right)n}, & \text{if $n$ is odd},\\[6pt]
\displaystyle \sqrt{\frac{(n-2)(n+2)}{12(n-1)}} = \displaystyle \sqrt{ \left(\frac{1}{12} +o(1)\right)n}, & \text{if $n$ is even.}
\end{cases}
\]
Thus, for large enough $n$, Theorem \ref{thm:BeckMakerCriterion} implies that Maker can force a win if
\[
b < \sqrt{\left(\frac{1}{12} +o(1)\right)n},
\]
and consequently, the bias threshold $b^*(n)$ satisfies the lower bound stated in \eqref{eqn:bias threshold inequality} from Theorem \ref{thm:bias-threshold}.

For the upper bound, our approach is to compare the directed triangle game to its undirected counterpart, and show that Breaker's win in the undirected triangle game implies Breaker's win in the directed triangle game as well.

Let
\[
X=\binom{V_n}{2},
\]
the set of unordered edges on the vertex set $V_n$.

Define the hypergraph
\[
\mathcal T_{\mathrm{und}}=(X,\mathcal F_{\mathrm{und}})
\]
for the undirected Maker-Breaker triangle game on $K_n$ by letting
$\mathcal F_{\mathrm{und}}$ consist of all $3$-subsets of $X$ that form an undirected triangle in $K_n$.

Define also
\[
\mathcal T_{\mathrm{dir}}=(X,\mathcal F_{\mathrm{dir}})
\]
for the directed triangle game on $\Pi(n)$ as follows:
for every $3$-subset $\{\{i,j\},\{j,k\},\{i,k\}\}\subset X$, we declare it to belong to
$\mathcal F_{\mathrm{dir}}$ if and only if the corresponding three directed edges in $\Pi(n)$ form a directed triangle.

With this identification of both games on the same ground set $X$, every directed triangle is in particular an undirected triangle. Hence
\[
\mathcal F_{\mathrm{dir}}\subseteq \mathcal F_{\mathrm{und}},
\]
that is,
\[
\mathcal T_{\mathrm{dir}}\subseteq \mathcal T_{\mathrm{und}}.
\]

We give below a lemma that allows us to `borrow' a condition for Breaker's win on the undirected triangle game to the premise of the directed triangle game on the same board, owing to a monotonicity argument common in Maker-Breaker games, and similar to the one we used in Lemma \ref{lem: breaker monotonicity}.

\begin{lem}\label{lem:transfer}
Let $\mathcal{H}=(X,\mathcal{F})$ and $\mathcal{H}'=(X,\mathcal{F}')$ be hypergraphs on the same vertex set such that $\mathcal{H}' \subseteq \mathcal{H}$. 
If Breaker has a winning strategy in the $(a:b)$ Maker-Breaker game on $\mathcal{H}$, then Breaker also has a winning strategy in the $(a:b)$ Maker-Breaker game on $\mathcal{H}'$.
\end{lem}

\begin{proof}
Assume for contradiction that Breaker has a winning strategy on $\mathcal{H}$, but Maker has one on $\mathcal{H}'$. 
Then, by following her winning strategy on $\mathcal{H}'$, Maker could ensure the complete claim of some hyperedge $F' \in \mathcal{F}'$. 
Since $\mathcal{F}' \subseteq \mathcal{F}$, the same set $F'$ also belongs to $\mathcal{F}$, so this strategy would yield a win for Maker on $\mathcal{H}$ as well. 
This contradicts the assumption that Breaker’s strategy on $\mathcal{H}$ is winning. 
Hence, Maker cannot have a winning strategy on $\mathcal{H}'$, and the claim follows.
\end{proof}

\begin{theo}[See \cite{glazik2022new}]
    In the biased $(1:b)$ Maker-Breaker undirected triangle game (with Maker starting first), Breaker has a winning strategy if \[b \ge \sqrt{\left(\frac{8}{3} + o(1)\right)n},\] for sufficiently large $n$.
\end{theo}

This result, along with Lemma \ref{lem:transfer}, implies that Breaker would also win the biased $(1:b)$ Maker-Breaker \emph{directed} triangle game, for $n$ large enough, whenever $b$ satisfies 
\begin{equation}
    b \ge \sqrt{\left(\frac{8}{3} + o(1)\right)n}.
\end{equation}
Consequently, the bias threshold $b^*(n)$ satisfies the upper bound in \eqref{eqn:bias threshold inequality} of Theorem \ref{thm:bias-threshold}.

\section{Proof of our result stated in Subsection \ref{subsec:randomtournament}}

We now consider the random tournament $T(n,p)$ from Definition \ref{defi:randomtournament}, where
$p\in(0,1)$ is fixed.
The cases $p=0$ and $p=1$ are transitive and hence trivially won by Breaker, so we restrict to $p\in(0,1)$.

Since Maker wins on $\Pi(7)$ by Theorem \ref{thm:thresholdunbiased}, it is enough to show that
$T(n,p)$ contains an induced copy of $\Pi(7)$ with probability tending to $1$.

\begin{lem}\label{lem:existence}
Let $T(n,p)$ be a random tournament with fixed $p\in(0,1)$.
Then
\[
\lim_{n\to\infty}\mathbb P\bigl(T(n,p)\text{ contains an induced copy of }\Pi(7)\bigr)=1.
\]
\end{lem}

\begin{proof}
For each $7$-element subset $A=\{v_1,\dots,v_7\}\subset V_n$ with
\[
v_1<v_2<\cdots<v_7,
\]
let $I_A$ be the indicator random variable of the event that the tournament induced by $T(n,p)$ on $A$
is exactly $\Pi(7)$ under the order-preserving identification $v_i\mapsto i$ for $i=1,\dots,7$.

Let
\[
X_n=\sum_{\substack{A\subset V_n\\ |A|=7}} I_A,
\]
so that $X_n$ counts the number of induced copies of $\Pi(7)$ in $T(n,p)$.

Among the pairs $1\le i<j\le 7$, exactly $12$ satisfy $i+j$ odd and exactly $9$ satisfy $i+j$ even.
Therefore, for every $A$,
\[
\mathbb P(I_A=1)=p^{12}(1-p)^9=:q>0.
\]
Hence
\[
\mathbb E[X_n]=\binom{n}{7}q=\Theta(n^7).
\]

We now estimate the second moment.
Write
\[
\mathbb E[X_n^2]
=
\sum_{\substack{A,B\subset V_n\\ |A|=|B|=7}}
\mathbb E[I_A I_B].
\]
For $0\le r\le 7$, let
\[
S_r
=
\sum_{\substack{A,B\subset V_n\\ |A|=|B|=7\\ |A\cap B|=r}}
\mathbb E[I_A I_B].
\]
Then
\[
\mathbb E[X_n^2]=\sum_{r=0}^7 S_r.
\]

If $r=0$, then $A$ and $B$ are disjoint, so the relevant edge orientations are independent. Thus
\[
\mathbb E[I_A I_B]=\mathbb E[I_A]\mathbb E[I_B]=q^2,
\]
and therefore
\[
S_0=\binom{n}{7}\binom{n-7}{7}q^2.
\]
Since
\[
\frac{\binom{n-7}{7}}{\binom{n}{7}}\to 1,
\]
it follows that
\[
S_0=(1+o(1))\binom{n}{7}^2 q^2=(1+o(1))(\mathbb E[X_n])^2.
\]

Now fix $1\le r\le 7$.
The number of ordered pairs $(A,B)$ of $7$-subsets with $|A\cap B|=r$ is
\[
\binom{n}{r}\binom{n-r}{7-r}\binom{n-7}{7-r}=O(n^{14-r}).
\]
Also,
\[
0\le \mathbb E[I_A I_B]\le 1.
\]
Hence
\[
S_r=O(n^{14-r})=o(n^{14})=o\bigl((\mathbb E[X_n])^2\bigr),
\]
because $(\mathbb E[X_n])^2=\Theta(n^{14})$.

Combining the estimates for $S_0,\dots,S_7$, we obtain
\[
\mathbb E[X_n^2]=(1+o(1))(\mathbb E[X_n])^2.
\]
Therefore
\[
\operatorname{Var}(X_n)
=
\mathbb E[X_n^2]-(\mathbb E[X_n])^2
=
o\bigl((\mathbb E[X_n])^2\bigr).
\]

By Chebyshev's inequality,
\[
\mathbb P(X_n=0)
\le
\mathbb P\bigl(|X_n-\mathbb E[X_n]|\ge \mathbb E[X_n]\bigr)
\le
\frac{\operatorname{Var}(X_n)}{(\mathbb E[X_n])^2}
\longrightarrow 0
\qquad\text{as }n\to\infty.
\]
Thus
\[
\mathbb P(X_n>0)\longrightarrow 1,
\]
which means that $T(n,p)$ contains an induced copy of $\Pi(7)$ with probability tending to $1$.
\end{proof}

\begin{proof}[Proof of Theorem \ref{thm:random}]
By Lemma \ref{lem:existence}, with probability tending to $1$, the tournament $T(n,p)$ contains an induced copy of $\Pi(7)$.
On that event, Maker restricts all her moves to this copy and follows her winning strategy from Theorem \ref{thm:thresholdunbiased}.
Thus the directed triangle game on $T(n,p)$ is a Maker's win with probability tending to $1$ as $n\to\infty$.
\end{proof}

A natural future direction is to study the model $T(n,p(n))$ and determine the threshold behaviour of the game when the orientation parameter varies with $n$.

\section{Proof of our result stated in Subsection \ref{subsec:flipbias}}

Traditionally, bias in Maker-Breaker games is introduced by allowing Breaker additional moves per turn. Here we study a different mechanism called the \emph{flip-bias}. In the flip-biased Maker-Breaker game, we allow Breaker to preemptively flip the orientations of at most $\kappa(n)$ edges from the parity tournament before Maker plays her first move.

Fix an integer $k\ge 0$. Breaker first chooses a set $S$ with $|S|\le k$ unordered pairs $\{i,j\}$ with $1\le i<j\le n$ and reverses the orientation of each corresponding edge of $\Pi(n)$. Let $\Pi(n)_S$ denote the resulting tournament. Maker and Breaker then play the usual unbiased Maker-Breaker directed triangle game on $\Pi(n)_S$ as in Definition \ref{defi: MB directed triangle game}.
We write $\kappa^*(n)$ for the least integer $k$ such that Breaker has some choice of $S$ with $|S|\le 
k$ for which Breaker wins on $\Pi(n)_S$.

\subsection{An upper bound on the flip-bias threshold}\label{subsec: upper bound on flip-bias}

We prove the upper bound in Theorem \ref{thm:flip-bias} by explicitly turning $\Pi(n)$ into a transitive tournament. Since a transitive tournament has no directed cycle, it has no directed triangle, and the resulting game is a trivial win for Breaker.

\begin{proof}[Upper bound in Theorem \ref{thm:flip-bias}]
Let us first address the case where $n$ is odd, that is $n = 2m+1$. 
Making a tournament transitive is equivalent to finding an ordering of its vertices such that every edge is oriented forward with respect to that ordering (i.e.\ it is oriented from the vertex having the smaller index to the vertex having the larger index).
Consider the following ordering of the vertices of the parity tournament $\Pi(n)$:
\[
\sigma:\quad 1,\ 2m,\ 2m-2,\ \ldots,\ 2,\ 2m+1,\ 2m-1,\ \ldots,\ 3.
\]

\begin{figure}[htbp]
    \centering
    \includegraphics[width=0.90\textwidth]{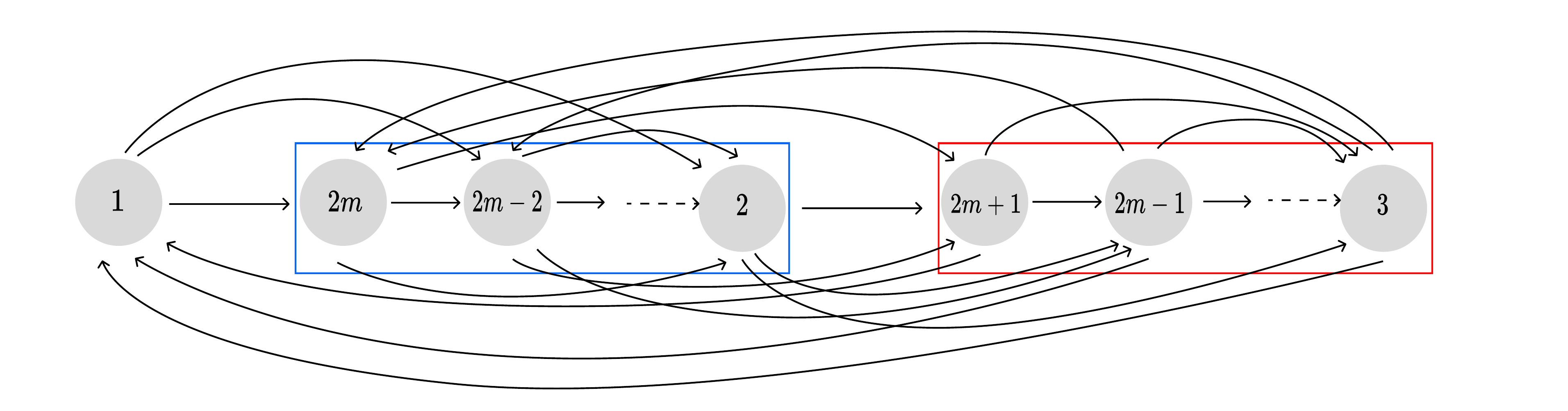}
    \caption{Parity tournament $\Pi(2m + 1)$: Edge orientations with respect to $\sigma$ before the flipping sequence}
    \label{fig:before_flips}
\end{figure}

We say that an edge from $i$ to $j$ \emph{points forward in $\sigma$}, if $i$ is indexed before $j$ in $\sigma$. We construct a set $S$ of edges of $\Pi(n)$ to flip so that every edge, after this flipping sequence, points forward in $\sigma$.

\noindent\textbf{Step 1.} For each odd vertex, say $o\in\{3,5,\dots,2m+1\}$, we flip the edge between $1$ and $o$ so that it now directs from $1$ to $o$. This would amount to exactly $m$ flips.

\noindent\textbf{Step 2.} For each even vertex $e=2a\in\{2,4,\ldots,2m\}$ and each odd $o\in\{3,5,\ldots,2a-1\}$, flip the edge between $e$ and $o$ so that it is now directed from $e$ to $o$. For a fixed $e=2a$, there are, $(a-1)$ eligible odd vertices $o$, hence in Step 2, we are required to flip $\sum_{a=1}^m (a-1)=\frac{m(m-1)}{2}$ edges.

Cumulatively, across both steps, the total number of edges flipped is
\[
|S|=m+\frac{m(m-1)}{2}=\frac{m(m+1)}{2}=\frac{n^2-1}{8}.
\]
We can readily verify that following these two steps, the resultant tournament is indeed transitive. Recall, from Definition \ref{defi:Parity rule}, that in $\Pi(n)$, edges between vertices of the same parity are oriented from the larger label to the smaller label, whereas edges between vertices of opposite parity are oriented from the smaller label to the larger label. The even vertices in $\sigma$ appear in a decreasing order and thus every edge between two even indexed vertices points forward by default. The odd vertices, barring $1$, i.e.\ the vertices in $\{3,5,\dots,2m +1\}$, appear in a decreasing order in $\sigma$ as well, and thus every between two odd indexed vertices points forward already. As for vertex $1$, the edges between $1$ and the even vertices point forward in $\sigma$ due to the parity rule for edge orientation, and edges between $1$ and the other odd vertices were addressed in the flips performed in Step 1. Thus, each edge from $1$ now points forward in $\sigma$. We are finally left with the edges between an even vertex $e$ and an odd vertex $o\ge 3$. Every even vertex appears before every such odd vertex in $\sigma$. If $e<o$ then $\Pi(n)$ already has the edge going from $e$ to $o$, that is forward in $\sigma$. If $o<e$ then $\Pi(n)$ has the edge directed from $o$ to $e$, and Step 2 flips exactly these edges to make them go from $e$ to $o$. Thus every even--odd edge points forward in $\sigma$ as well.
Therefore every edge of the perturbed tournament is oriented forward with respect to $\sigma$, so the perturbed tournament is transitive. Hence it contains no directed triangle, and Breaker wins. This proves that $\kappa^*(n)\le (n^2-1)/8$ for odd $n$.

The case of even $n$ is handled in the same way. Writing $n=2m$, consider the order
\[
\sigma:\quad 1,\ 2m,\ 2m-2, \dots,\ 2,\ 2m-1,\ 2m-3,\ \ldots,\ 3,
\]
and perform Step 1 only for odd $o\in\{3,5,\ldots,2m-1\}$ together with the same Step 2. The same parity checks show that the resulting tournament is transitive. The number of flipped edges is
\[
(m-1)+\sum_{a=1}^m (a-1)=(m-1)+\frac{m(m-1)}{2}=\frac{n^2}{8}+\frac{n}{4}-1.
\]
In particular, for all $n$ we obtain the asymptotic upper bound
\begin{equation}\label{eqn:upper}
    \kappa^*(n)\le \frac{n^2}{8}+O(n).
\end{equation}
\end{proof}

\begin{figure}[htbp]
    \centering
    \includegraphics[width=0.90\textwidth]{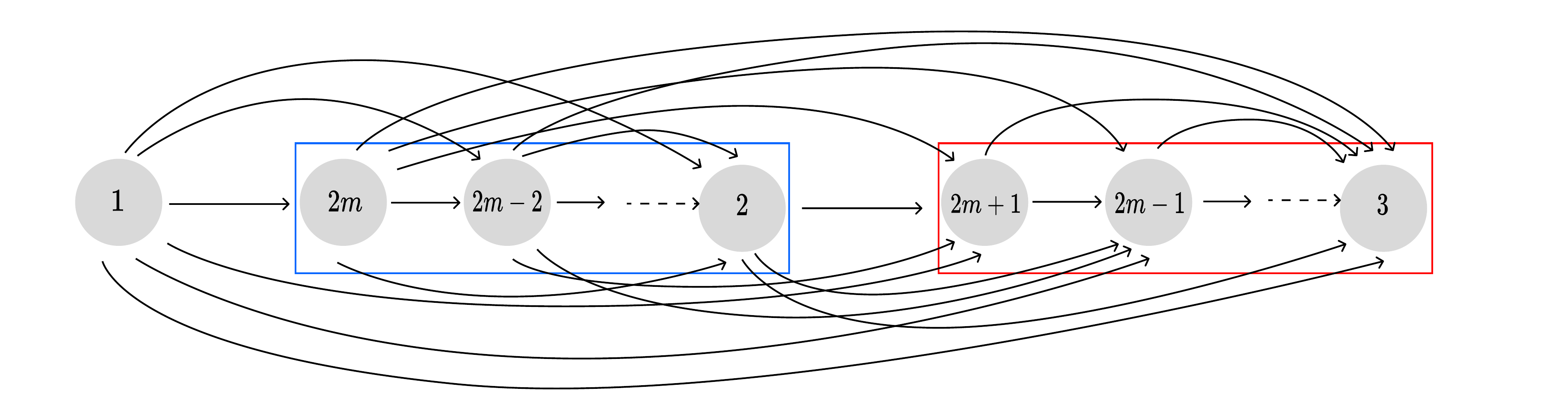}
    \caption{Transitive tournament $\Lambda(2m +1)$: Edge orientations with respect to $\sigma$ after the flipping sequence}
    \label{fig:after_flips}
\end{figure}

\begin{proof}[Proof of Proposition \ref{prop:linear-triangles-breakerwin}]
Start with a transitive tournament whose vertices are ordered
\[
v_1,v_2,\dots,v_n,
\]
so that every edge points forward with respect to this order.

Partition this sequence into consecutive disjoint blocks of size $3$, ignoring the final block if it has size $1$ or $2$.
For each full block
\[
(v_{3t+1},v_{3t+2},v_{3t+3}),
\]
reverse the single edge between $v_{3t+1}$ and $v_{3t+3}$.
Inside that block, the three edges now form a directed triangle.

No directed triangle uses vertices from more than one block.
Indeed, all edges between different blocks still point forward with respect to the original linear order, so any triple of vertices meeting at least two blocks has a global source and is therefore transitive.
Hence the only directed triangles are the ones created inside the full blocks, and there are exactly $\lfloor n/3\rfloor$ of them.

Breaker now uses the following response strategy.
Whenever Maker first claims an edge from some block, Breaker immediately claims a different edge from that same block.
Since each block contains exactly one directed triangle and the blocks are edge-disjoint, Maker can never claim all three edges of the triangle in any block.
Therefore Breaker wins.
\end{proof}

\subsection{A lower bound on the flip-bias threshold}\label{subsec: lower bound on flip-bias}

We now prove the lower bounds in Theorem \ref{thm:flip-bias}.
We use the following sufficient condition for Maker's win.

\begin{theo}[See \cite{beck1981positional}]\label{thm:beck}
Let $(X,\mathcal{F})$ be an $s$-uniform hypergraph and define
\[
\Delta_2(\mathcal{F})
=
\max\Bigl\{
\,|\{F\in\mathcal{F}:\{x,y\}\subset F\}| : x,y\in X,\ x\neq y
\Bigr\}.
\]
If
\begin{equation}
|\mathcal{F}|>2^{\,s-3}\,\Delta_2(\mathcal{F})\,|X|,\label{Maker's_win_Beck}
\end{equation}
then Maker has a winning strategy in the Maker-Breaker game on $(X,\mathcal{F})$.

\end{theo}

In the directed triangle game on a tournament $T(n)$, the ground set is $X=E(T(n))$ and
$\mathcal F$ is the family of directed triangles of $T(n)$.
Since the game is $3$-uniform and any two directed edges lie together in at most one directed triangle, we have
$\Delta_2(\mathcal F)=1$.
Therefore Theorem \ref{thm:beck} immediately yields the following.

\begin{lem}\label{lem:many-triangles-maker-win}
If a tournament $T(n)$ has more than $\binom{n}{2}$ directed triangles, then Maker wins the directed triangle game on $T(n)$.

\end{lem}
\begin{proof}
Apply Theorem \ref{thm:beck} with $s=3$, $X=E(T(n))$, and $\mathcal F$ equal to the family of directed triangles of $T(n)$. In this setting, any two directed edges belong to at most one directed triangle, so $\Delta_2(\mathcal F)=1$, and $|X|=\binom{n}{2}$. Hence the criterion \eqref{Maker's_win_Beck} becomes
\[
|\mathcal F|>|X|=\binom{n}{2}.
\]
Therefore, if $T(n)$ has more than $\binom{n}{2}$ directed triangles, then Maker has a winning strategy.
\end{proof}

We next express the number of directed triangles in terms of the score sequence.

\begin{lem}\label{lem:triangle-count-by-scores}
Let $T(n)$ be a tournament on $n$ vertices, and let $d_{T(n)}^+(v)$ denote the out-degree of a vertex $v$ in $T(n)$.
Define
\[
\delta_{T(n)}(v)=d_{T(n)}^+(v)-\frac{n-1}{2}.
\]
Then the number $c_3(T(n))$ of directed triangles in $T(n)$ satisfies
\[
c_3(T(n))
=
\binom{n}{3}
-
\sum_v \binom{d_{T(n)}^+(v)}{2}
=
\frac{n(n^2-1)}{24}
-\frac12\sum_v \delta_{T(n)}(v)^2.
\]
\end{lem}

\begin{proof}
Every $3$-subset of vertices either forms a directed triangle or a transitive triangle.
A transitive triangle has a unique source, namely the unique vertex of out-degree $2$ inside that triangle.
Hence the number of transitive triangles equals
\[
\sum_v \binom{d_{T(n)}^+(v)}{2},
\]
which proves the first identity.

For the second identity, expand
\[
\binom{d_{T(n)}^+(v)}{2}
=
\frac12 d_{T(n)}^+(v)\bigl(d_{T(n)}^+(v)-1\bigr),
\]
use
\[
\sum_v d_{T(n)}^+(v)=\binom{n}{2},
\]
and write
\[
d_{T(n)}^+(v)=\frac{n-1}{2}+\delta_{T(n)}(v).
\]
Since
\[
\sum_v \delta_{T(n)}(v)=0,
\]
we obtain
\[
\sum_v \binom{d_{T(n)}^+(v)}{2}
=
\frac{n(n-1)(n-3)}{8}
+
\frac12\sum_v \delta_{T(n)}(v)^2.
\]
Subtracting this from $\binom{n}{3}$ gives
\[
c_3(T(n))
=
\frac{n(n^2-1)}{24}
-
\frac12\sum_v \delta_{T(n)}(v)^2,
\]
as claimed.

\end{proof}

We now control how much the quadratic term can change after a bounded number of flips.

\begin{lem}\label{lem:delta-sq-odd}
Let $n=2m+1$ be odd and let $S$ be any set of flipped edges with $|S|\le k$. Let $\Pi(n)_S$ denote the resultant tournament after Breaker flips these $S$ edges.
Then the tournament $\Pi(n)_S$ has at least
\[
\frac{n(n^2-1)}{24}-mk
\]
directed triangles.

\end{lem}

\begin{proof}
Write $T=\Pi(n)_S$.
Since $\Pi(n)$ is regular when $n$ is odd, every vertex has out-degree $m$ in $\Pi(n)$.
Set
\[
x(v)=d_T^+(v)-m.
\]
Since $(n-1)/2=m$, we have
\[
\delta_T(v)=d_T^+(v)-\frac{n-1}{2}=d_T^+(v)-m=x(v)
\]
for every vertex $v$.

Each flip reverses exactly one edge, so it changes the out-degrees of exactly two vertices by $\pm1$.
Therefore after the $S$ edges have been flipped, we get
\[
\sum_v |x(v)|\le 2|S|\le 2k.
\]
Also, since $0\le d_T^+(v)\le 2m$, we have $|x(v)|\le m$ for all $v$.
Hence
\[
\sum_v x(v)^2\le m\sum_v |x(v)|\le 2mk.
\]
Applying Lemma \ref{lem:triangle-count-by-scores}, we obtain
\[
c_3(T)
=
\frac{n(n^2-1)}{24}
-
\frac12\sum_v x(v)^2
\ge
\frac{n(n^2-1)}{24}
-mk.
\]
\end{proof}

\begin{lem}\label{lem:delta-sq-even}
Let $n=2m$ be even and let $S$ be any set of flipped edges with $|S|\le k$. Let $\Pi(n)_S$ denote the resultant tournament after Breaker flips these $S$ edges.
Then the tournament $\Pi(n)_S$ has at least
\[
\frac{n^3-4n}{24}-(m+1)k
\]
directed triangles.
\end{lem}

\begin{proof}
Write $T=\Pi(n)_S$.
Let $d_0^+(v)$ and $d_T^+(v)$ denote the out-degrees of $v$ in $\Pi(n)$ and $T$, respectively.
Set
\[
\delta_0(v)=d_0^+(v)-\frac{n-1}{2},
\quad
\delta(v)=d_T^+(v)-\frac{n-1}{2},
\quad
x(v)=d_T^+(v)-d_0^+(v).
\]
Then
\[
\delta(v)=\delta_0(v)+x(v).
\]

Since $\Pi(n)$ is near-regular when $n$ is even, every vertex has out-degree either $m$ or $m-1$.
Thus, for all $v$,
\[
\delta_0(v)\in\left\{-\frac12,\frac12\right\}
\]

Each flip changes two out-degrees by $\pm1$, so
\[
\sum_v x(v)=0,
\qquad
\sum_v |x(v)|\le 2|S|\le 2k.
\]
Moreover, for all $v$,
\[
|x(v)|\le m
\]
because $d_0^+(v)\in\{m-1,m\}$ and $0\le d_T^+(v)\le 2m-1$.

Now
\[
\sum_v \delta(v)^2-\sum_v \delta_0(v)^2
=
2\sum_v \delta_0(v)x(v)+\sum_v x(v)^2.
\]
Using $|\delta_0(v)|\le 1/2$ and $\sum_v |x(v)|\le 2k$, we get
\[
\sum_v \delta_0(v)x(v)\le \frac12\sum_v |x(v)|\le k.
\]
Also,
\[
\sum_v x(v)^2\le m\sum_v |x(v)|\le 2mk.
\]
Therefore
\[
\frac12\left(\sum_v \delta(v)^2-\sum_v \delta_0(v)^2\right)
\le k+mk=(m+1)k.
\]

By Lemma \ref{lem:triangle-count-by-scores},
\[
c_3(T)
=
\frac{n(n^2-1)}{24}
-\frac12\sum_v \delta(v)^2.
\]
Since $\delta_0(v)\in\{-1/2,1/2\}$ for all $v$, we have
\[
\sum_v \delta_0(v)^2=\frac n4,
\]
and hence
\[
\frac{n(n^2-1)}{24}-\frac12\sum_v \delta_0(v)^2
=
\frac{n(n^2-1)}{24}-\frac n8
=
\frac{n^3-4n}{24}.
\]
Combining the two expressions, we get
\[
c_3(T)\ge \frac{n^3-4n}{24}-(m+1)k,
\]
as claimed.
\end{proof}

We can now complete the lower bounds in Theorem \ref{thm:flip-bias}.

\begin{proof}[Proof of the lower bounds in Theorem \ref{thm:flip-bias}]
Suppose first that $n=2m+1$ is odd.
By Lemma \ref{lem:delta-sq-odd}, every tournament $\Pi(n)_S$ with $|S|\le k$ has at least
\[
w(n)-mk
\]
directed triangles.
If
\[
k<\frac{n(n-11)}{12}
=
\frac{(2m+1)(m-5)}{6},
\]
then
\[
w(n)-mk
>
\frac{(2m+1)m(m+1)}{6}
-
m\cdot \frac{(2m+1)(m-5)}{6}
=
m(2m+1)
=
\binom{n}{2}.
\]
Hence Lemma \ref{lem:many-triangles-maker-win} implies that Maker wins on $\Pi(n)_S$ for every such $S$.
Therefore
\[
\kappa^*(n)\ge \frac{n(n-11)}{12}
\qquad\text{for odd }n\ge 11.
\]

Now suppose that $n=2m$ is even.
By Lemma \ref{lem:delta-sq-even}, every tournament $\Pi(n)_S$ with $|S|\le k$ has at least
\[
w(n)-(m+1)k
\]
directed triangles.
If
\[
k\le \frac{n(n-14)}{12}
=
\frac{m(m-7)}{3},
\]
then
\[
w(n)-(m+1)k
\ge
\frac{m^3-m}{3}
-
(m+1)\frac{m(m-7)}{3}
=
2m^2+2m
=
\binom{n}{2}+3m
>
\binom{n}{2}.
\]
Again Lemma \ref{lem:many-triangles-maker-win} implies that Maker wins.
Hence
\[
\kappa^*(n)\ge \frac{n(n-14)}{12} + 1
\qquad\text{for even }n\ge 14.
\]

Together with the upper bound from Subsection \ref{subsec: upper bound on flip-bias}, this completes the proof of Theorem \ref{thm:flip-bias}.
\end{proof}

\section*{Acknowledgments}
We are grateful to Prof.\ Michael Krivelevich for his insightful comments and for guiding us toward the relevant literature. This work was carried out as part of the first author's Master's thesis project at the Indian Institute of Science Education and Research Pune.
\printbibliography
\end{document}